\documentclass{elsarticle}

\usepackage{lineno,hyperref}
\usepackage{amsmath}
\usepackage{amsfonts}
\usepackage{amssymb}
\usepackage{graphicx}
\usepackage{hyperref}
\usepackage{amsthm}
\usepackage{bookmark}
\usepackage{color}
\usepackage{xcolor}
\usepackage{caption}
\usepackage{multirow}
\usepackage[T1]{fontenc}
\usepackage[utf8]{inputenc}

\definecolor{carminepink}{rgb}{0.92, 0.3, 0.26}
\hypersetup{colorlinks=true,allcolors=carminepink} 

\journal{Journal of Mathematical Analysis and Applications}

\newtheorem{theorem}{Theorem}[section]
\newtheorem*{theorem*}{Theorem 1.1}
\newtheorem*{acknow*}{Acknowledgement}
\newtheorem{corollary}[theorem]{Corollary}

\newtheorem{lemma}[theorem]{Lemma}

\newtheorem{proposition}[theorem]{Proposition}
\newtheorem{remark}{Remark}[section]
\DeclareSymbolFont{bbold}{U}{bbold}{m}{n}
\DeclareSymbolFontAlphabet{\mathbbold}{bbold}

\begin{document}
	
	\begin{frontmatter}	
	\title{Tail Behaviour of Mexican Needlets}

	\author[mainaddress,secondary]{Claudio 		Durastanti\fnref{myfootnote}}
		\ead{claudio.durastanti@gmail.com}
		\address[mainaddress]{Dipartimento di Matematica, Università di Roma ``Tor Vergata"}
		\address[secondary]{Fakult\"at f\"ur Matematik, Ruhr Universit\"at, Bochum}
		\fntext[myfootnote]{This research is supported by the ERC Grant n. 277742
			Pascal and by the German DFG grant GRK 2131.} 
			
		\begin{abstract}
			\noindent 
		In this paper we study the tail behaviour of Mexican needlets, a class of
		spherical wavelets introduced by Geller and Mayeli in \cite{gm2}. More specifically, we provide an
		explicit upper bound depending on the resolution level $j$ and a parameter $s$ governing the shape of the Mexican needlets.
		\end{abstract}
		
		\begin{keyword}
			Wavelets \sep Mexican needlets \sep
			 sphere \sep concentration properties \sep spatial localization.
			\MSC[2010] 42C40 \sep 42C10 \sep 42C15.
		\end{keyword}
		
	\end{frontmatter}

\section{Introduction}

A lot of interest has recently been focussed on various forms of spherical
wavelets (see, for example, \cite{ant, da,fred,hol,npw1,wiaux} and the references therein). This attention has also
been fuelled by strong applied motivations, for instance in Astrophysics and
Cosmology. More specifically, we refer to spherical Mexican hat
wavelets (see, for instance, \cite{mcewen}), axisymmetric, directional, and steerable wavelets (cf., for example, \cite{mcewen2,mcewen3,mcewen4,wiaux}), ridgelets and curvelets (see, for instance, \cite{mcewen5,starck}). 

Many theoretical and applied papers have been concerned, in particular, with the so-called spherical
needlets, which were introduced into the Functional Analysis literature by 
\cite{npw1,npw2}. Loosely speaking, the latter can be envisaged as a
convolution of the spherical harmonics with a weight function which is
smooth and compactly supported in the harmonic domain (more details will be
given below). Localization properties in this framework were fully
investigated by \cite{npw1,npw2}.
Needlets have been recently generalized to various directions. Spin and mixed needlets were constructed over spin fiber bundles in \cite{gmspin,gmmixed}, respectively. Needlet-like wavelets were also developed on the unitary ball in \cite{tredneed,petruxu} and over compact manifolds in \cite{knp}. This framework has been also extended to allow for an
unbounded support in the frequency domain by \cite{gm2}, see also \cite{gm1,gm3}; the latter construction is usually labelled as Mexican needlets.
Examples of applications, mainly related to the study of Cosmic Microwave Background (CMB) radiation, can
be found in \cite{cammar,dll,lanmar2,mayeli,scodeller}. 

As described in details below, Mexican needlets enjoy
excellent localization properties in the real domain; in this paper, we
investigate the relationship between the tail decay and the exact shape of
the weight function. Indeed, the aim of this work is to provide analytic
expressions to bound the tail behaviour in the real domain. We prove here that the
tails are Gaussian up to a polynomial term, whose dependence on the choice
of the kernel can be identified explicitly. \\More specifically, for any (multi)resolution level $j$, consider a partition of the sphere into a set (of cardinality $N_j$) of spherical subregions of area $\lambda_{jk}$ and midpoint $\xi_{jk}$, $k=1,\ldots,N_j$. For any $k$, let $\vartheta:=\vartheta_{jk}\left(x\right)$ denote the geodesic distance between a generic coordinate $x\in \mathbb{S}^2$ and $\xi_{jk}$. For the scale parameter  $B>1$ and the shape parameter $s\in \mathbb{N}$, we shall consider wavelet filters of the form 
\begin{equation*}
\Psi _{jk ;s}\left( \vartheta \right) :=
\frac{\sqrt{\lambda_{jk}}}{2 \pi}\sum_{\ell=0}^{\infty }\left( \left( \frac{ \ell+\frac{1}{2}}{B^j}\right) \right)
^{2s}\exp \left( -\left( \frac{\ell+\frac{1}{2}}{B^j}\right) ^{2}\right) \left(
2\ell+1\right) P_{\ell}\left( \cos \vartheta \right),%
\end{equation*}%
where $P_{\ell}\left( \cdot \right)$ denotes the standard Legendre polynomial
of degree $\ell$. In Theorem \ref{theo:loc}, we shall be able to show that 
\begin{equation*}
\left\vert \Psi _{jk ;s}\left( \vartheta \right) \right\vert \leq
C_{s}B^j e^{-\left( 2B^j\varepsilon \right) ^{2}}\left( 1+\left\vert H_{2s}\left( B^j\vartheta\right) \right\vert \right),
\end{equation*}%
where $C_{s}$ is a positive constant and $H_{2s}\left( \cdot \right) $ identifies the Hermite polynomial of
degree $2s$.

It is important to remark that in \cite{gm2} the authors obtained an
analogous expression for the $n$-dimensional sphere, limiting their
investigation to the case of the shape parameter $s=1$, which can be linked to the spherical Mexican hat wavelets (see Remark \ref{scodellerlink} and \cite{scodeller}). In this paper, we
will extend this bound for any choice of $s\in \mathbb{N}$. 
Our argument
exploits a technique similar to the one used by Narcowich, Petrushev and
Ward in \cite{npw1} (see also \cite{npw2} and the textbook \cite[Section 13.3]{marpecbook}). Furthermore, an analogous method was developed in \cite{mcewen4} to establish concentration properties of spin directional wavelets.   
In our proof, we will also exploit the analytic form of the weight function
to compute exactly its Fourier transform in terms of Hermite polynomials;
this will also allow us to investigate explicitly the roles of the
resolution level $j$ and of the shape parameter $s$. We also establish bounds on the $L^p$-norms of the Mexican needlets, depending on the resolution level $j$ and on the scale parameter $B$ (see Corollary \ref{coronorms}). Furthermore, in Proposition \ref{delta_s} we provide an explicit connection between Mexican needlets with different shape by means of the spherical Laplacian operator.

The plan of this paper is as follows. In Section \ref{sec:costr} we recall
the definition and some pivotal properties of Mexican needlets; in Section %
\ref{sec:loc} we exploit our main theorem while Section \ref{sec:aux}
collects some auxiliary results.

\section{The construction of Mexican needlets\label{sec:costr}}

In this Section we shall review Mexican needlets, as developed by Geller and
Mayeli, see \cite{gm1,gm2,gm3}. As already mentioned and
similarly to standard needlets (cf. \cite{npw1,npw2}), Mexican
needlets can be viewed as a combination of Legendre polynomials weighted by
a smooth window function.\\ On one hand, recall the well-known decomposition of $L^{2}\left( \mathbb{S}^{2}\right)$, the space of the square-integrable functions over the sphere, given by 
\begin{equation*}
L^{2}\left( \mathbb{S}^{2}\right) =\bigoplus_{\ell\geq 0}H_{\ell},
\end{equation*}%
where $H_{\ell}$ is the space of the homogeneous polynomials of degree $\ell$,
spanned by the spherical harmonics $\left\{ Y_{\ell m},m=-\ell,\ldots,\ell\right\}$. Therefore, spherical harmonics provide an orthonormal basis for $L^2\left(\mathbb{S}^2\right)$. For further details, the reader is referred, for example, to the textbook \cite{steinweiss}); here we just recall the so-called summation formula, i.e., for any $\ell \geq 0$ and for any $x,y \in \mathbb{S}^2$, 
\begin{equation}\label{summation}
\sum_{m=-\ell}^{\ell} \overline{Y}_{\ell m} \left(x\right)Y_{\ell m}\left(y\right)=\frac{2\ell+1}{4\pi}P_{\ell}\left(\langle x,y\rangle\right), 
\end{equation}
where $\langle \cdot,\cdot\rangle$ denotes the geodesic distance over the sphere and $P_{\ell}\left(\cdot\right)$ is the Legendre polynomial of order $\ell$, given by  
\begin{equation*}
P_{\ell}\left( u\right) :=\frac{1}{2^{\ell}\ell!}\frac{d^{\ell}}{du^{\ell}}\left(
u^{2}-1\right) ^{\ell}, \quad u\in \left[ -1,1\right],
\end{equation*}%
see, for example, \cite[ Chapter 22, Eqq. (22.1.6) and
(22.2.10)]{abramostegun}.
On the other hand, consider the window (or weight) function $f_{s}:\mathbb{R}\rightarrow 
\mathbb{R}^{+}$ 
\begin{equation}
f_{s}\left( t\right) :=t^{2s}e^{-t^{2}}, \quad t\in \mathbb{R},
\label{weight}
\end{equation}%
for $s\in \mathbb{N}$, so that, for any $t\in \mathbb{R}$, we have that  
\begin{equation*}
0<\frac{m_{B;s}}{\log B}\leq \sum_{j=-\infty }^{\infty }f_{s}^{2}\left( 
\frac{t}{B^{j}}\right) \leq \frac{M_{B;s}}{\log B}<\infty,
\label{emme}
\end{equation*}%
where%
\begin{align*}
m_{B;s} &:=\eta _{s}\left( 1-O\left( \left\vert \left( \frac{B^{\frac{1}{2}%
}-1}{B^{\frac{1}{2}}}\right) ^{2}\log \left( \frac{B^{\frac{1}{2}}-1}{B^{%
\frac{1}{2}}}\right) \right\vert \right) \right) \text{ ,} \\
M_{B;s} &:=\eta _{s}\left( 1+O\left( \left\vert \left( \frac{B^{\frac{1}{2}%
}-1}{B^{\frac{1}{2}}}\right) ^{2}\log \left( \frac{B^{\frac{1}{2}}-1}{B^{%
\frac{1}{2}}}\right) \right\vert \right) \right) \text{ ,}
\end{align*}%
as $B\rightarrow 1$, $B>1$, and 
\begin{equation*}
\eta _{s}:=\int_{0}^{\infty }f_{s}^{2}\left( t\ell\right) \frac{dt}{t}=\frac{%
\Gamma \left( 2s\right) }{2^{2s+1}}, \label{eta}
\end{equation*}
see also see also \cite{gm2}.

In \cite{gm2} (see also \cite{gm1,gm3}), it was proven that, for
any given resolution level $j\in \left( -\infty ,\infty \right) $, there exists a finite
set of measurable subregions of the sphere $\left\{ E_{jk}\right\} _{k=1}^{N_{j}}$, of
diameter $\rho _{jk}$ and area $\lambda _{jk}$, such that 
\begin{align*}
& \cup _{k=1}^{N_{j}}E_{jk}=\mathbb{S}^{2},  \notag \\
& E_{jk_{1}}\cap E_{jk_{2}} =\emptyset, \text{ for any }k_{1}\neq k_{2},\notag \\
& \rho _{jk} \leq c_{B}B^{-j},  \label{diam}
\end{align*}%
where $c_{B}>0$, $B>1$. Each of these regions can be indexed by a point $\xi
_{jk}\in E_{jk}$, typically chosen as its midpoint. For $x,y\in \mathbb{S}^{2}$, let $W_{j;s}:\mathbb{S}^{2}\times \mathbb{S}%
^{2}\rightarrow \mathbb{R}$ be defined as follows %
\begin{equation*}
W_{j;s}\left( x,y\right) :=\sum_{\ell\geq 0}f_{s}\left( \frac{\ell+\frac{1}{2}}{%
B^{j}}\right) \frac{\left( \ell+\frac{1}{2}\right) }{2\pi }P_{\ell}\left(
\left\langle x,y\right\rangle \right).
\end{equation*}%
Consider now the kernel operator $K_{j;s}$ on $L^{2}\left( \mathbb{S}%
^{2}\right) $: 
\begin{equation*}
K_{j;s}F\left( x\right) =\int_{S^{2}}W_{j;s}\left( x,y\right) F\left(
y\right) dy, \quad  F\in L^{2}\left( \mathbb{S}^{2}\right) \text{.}
\end{equation*}%
For $\delta _{0}>0$ sufficiently small, it is shown in \cite{gm2} that if, for any $k$ and for any $j$ so that $%
c_{B}B^{-j}<\delta _{0}$, the area $%
\lambda _{jk}$ is comparable with $\left( c_{B}B^{-j}\right) ^{2}$, then, for $%
\varepsilon >0$, 
\begin{equation*}
\left( m_{B;s}-\varepsilon \right) \left\Vert F\right\Vert _{L^{2}\left( 
\mathbb{S}^{2}\right) }^{2}\leq \sum_{j=-\infty }^{\infty
}\sum_{k=1}^{N_{j}}\lambda _{jk}\left\vert K_{j;s}F\left( \xi _{jk}\right)
\right\vert ^{2}\leq \left( M_{B;s}+\varepsilon \right) \left\Vert
F\right\Vert _{L^{2}\left( \mathbb{S}^{2}\right) }^{2}.
\end{equation*}

\begin{remark}
Let us introduce preliminarily the notation $a\approx b$ if there exist $c^{\prime },c^{\prime \prime }>0$ so
that $c^{\prime }b\leq a\leq c^{\prime \prime }b$. In many practical
applications, the set of $E_{jk}$, labelled by the pair $\left( \xi
_{jk},\lambda _{jk}\right) $, can be identified with those evaluated by
common packages such as HealPix (see for instance \cite{healpix}), where,  for any $j$, 
$\lambda _{jk}\approx 4\pi \rho _{jk}^{2}$. Partitioning the sphere
into $N_{j}\approx B^{2j}$ regions $E_{jk}$, we have that 
\begin{equation*}
\mu \left( \cup
_{k=1}^{N_{j}}E_{jk}\right) =4\pi =\mu \left( \mathbb{S}^{2}\right),
\end{equation*}
where $\mu $ denotes the area. Hence, for the sake of simplicity, we
consider, for any $j$, 
\begin{align*}
\lambda _{jk} \approx &B^{-2j}\text{, }k=1,...,N_{j},
\label{lambdasize} \\
N_{j} \approx &B^{2j}. 
\end{align*}
\end{remark}

Let us now define $\Psi _{jk;s}:\mathbb{S}^{2}\rightarrow \mathbb{R}$ as%
\begin{equation}
\Psi _{jk;s}\left( x\right) :=\sqrt{\lambda _{jk}}K_{j;s}\left( x,\xi
_{jk}\right) \text{ .}  \label{mex1}
\end{equation}%
An alternative form of Mexican needlets can be given by 
\begin{align}
\notag \psi _{jk;s}\left( x\right) :=\sqrt{\lambda _{jk}}\sum_{\ell=0}^{\infty
}f_{s}\left( \frac{\sqrt{-e_{\ell}}}{B^{j}}\right) \sum_{m=-\ell}^{\ell}\overline{Y}%
_{\ell m}\left( \xi _{jk}\right) Y_{\ell m}\left( x\right)\\
\sqrt{\lambda _{jk}}\sum_{\ell=0}^{\infty
}f_{s}\left( \frac{\sqrt{-e_{\ell}}}{B^{j}}\right) \frac{2 \ell +1}{4\pi} P_{\ell}\left(\langle x, \xi_{jk}\rangle \right),
\label{mexicantot}
\end{align}%
$\left\{ e_{\ell}\right\} $ denoting the spectrum of the spherical Laplacian $%
\Delta _{\mathbb{S}^{2}}$ associated to the eigenfunctions $\left\{
Y_{\ell m}\right\} $, i.e., $e_{\ell}=-\ell(\ell+1),$ whence%
\begin{equation*}
\left( \Delta _{\mathbb{S}^{2}}-e_{\ell}\right) Y_{\ell m}\left( x\right) =0.
\end{equation*}%
Note that the last equality in \eqref{mexicantot} is due to \eqref{summation}.
\begin{remark}\label{remark:equivalence} \eqref{mexicantot} is the standard definition of Mexican needlets in the literature, while Theorem \ref{theo:loc} is focussed on \eqref{mex1}.
Consider, anyway, that the difference between \eqref{mex1} and \eqref{mexicantot} concerns only  the argument of $f_{s}\left( \cdot \right) $. In the former, the argument is given by the square root of the eigenvalue of the Laplacian operator $-e_{\ell}$, while in the latter it is replaced
by $\ell+\frac{1}{2}$. This formulation is instrumental for the derivation of the
localization property in Section \ref{sec:loc} (cf., more specifically, \eqref{Kappa}), as in \cite{npw1} for the
standard needlet case (see also Section 13.3 in the Appendix of \cite%
{marpecbook}). This is a minor difference, asymptotically negligible
considering that, trivially,%
\begin{equation*}
\lim_{j\rightarrow \infty }\lim_{\ell\rightarrow \infty } \frac{f_{s}\left( \frac{\sqrt{-e_{\ell}}}{B^j} \right) }{f_{s}\left(\frac{ \ell+\frac{1}{2}}{B^j} \right) }=\lim_{j\rightarrow \infty }\lim_{\ell\rightarrow \infty }\frac{\left( \frac{%
\ell\left( \ell+1\right) }{B^{2j}}\right) ^{s}\exp \left( -\frac{\ell\left(
\ell+1\right) }{B^{2j}}\right) }{\left( \frac{\left(\ell+\frac{1}{2}\right) ^{2}}{%
B^{2j}}\right) ^{s}\exp \left( -\frac{\left( \ell+\frac{1}{2}\right) ^{2}}{%
B^{2j}}\right) }=1\text{ .}
\end{equation*}%
As a consequence, for any $x\in\mathbb{S}^2$, we obtain an asymptotic equivalence between $\Psi _{jk;s}\left(x\right)$ and $\psi _{jk;s}\left(x\right)$.  Hence, the localization property, proved in Theorem \ref{theo:loc} for \eqref{mex1}, holds also for the Mexican needlets given by \eqref{mexicantot}.
\end{remark}
For $F \in L^2\left(\mathbb{S}^2\right)$ and for any $j,k$, let the Mexican needlet coefficients be given by
\begin{equation*}
\beta _{jk;s}:=\left\langle F,\psi _{jk;s}\right\rangle _{L^{2}\left( 
\mathbb{S}^{2}\right) }\text{ .}  \label{beta}
\end{equation*}
It is proven in \cite{gm2} that there exists a constant $C_{0}=C_{0}\left(
B,c_{B},f_{s}\right) $ such that 
\begin{equation*}
\left( m_{B;s}-C_{0}\right) \left\Vert F\right\Vert _{L^{2}\left( \mathbb{S}%
^{2}\right) }^{2}\leq \sum_{j=-\infty }^{\infty
}\sum_{k=1}^{N_{j}}\left\vert \beta _{jk}\right\vert ^{2}\leq \left(
M_{B;s}+C_{0}\right) \left\Vert F\right\Vert _{L^{2}\left( \mathbb{S}%
^{2}\right) }^{2}\text{ .}
\end{equation*}%
Hence, if $m_{B;s}-C_{0}>0$, $\left\{ \psi _{jk;s}\right\} $ is a frame for $%
L^{2}\left( \mathbb{S}^{2}\right) $ bounded by $\left( m_{B;s}-C_{0}\right) 
$ and $\left( M_{B;s}+C_{0}\right) $. It holds that 
\begin{equation*}
\frac{M_{B;s}+C_{0}}{m_{B;s}-C_{0}}\sim \frac{M_{B;s}}{m_{B;s}}=1+O\left(
\left\vert \left( \frac{B^{\frac{1}{2}}-1}{B^{\frac{1}{2}}}\right) ^{2}\log
\left( \frac{B^{\frac{1}{2}}-1}{B^{\frac{1}{2}}}\right) \right\vert \right) 
\text{ ,}
\end{equation*}%
as $B\rightarrow 1,B>1$, and that 
\begin{equation}
\sum_{j=-\infty }^{\infty }\sum_{k=1}^{N_{j}}\left\vert \beta
_{jk;s}\right\vert ^{2}=\frac{\eta _{s}\left( 1+\delta \right) }{\log B}%
\left\Vert F\right\Vert _{L^{2}\left( \mathbb{S}^{2}\right) }^{2}\text{ ,} 
\notag
\end{equation}%
where $\delta :=\delta \left( B\right) =O\left( \left\vert \left( \frac{B^{%
\frac{1}{2}}-1}{B^{\frac{1}{2}}}\right) ^{2}\log \left( \frac{B^{\frac{1}{2}%
}-1}{B^{\frac{1}{2}}}\right) \right\vert \right) $ as $B\rightarrow 1,B>1$,
is such that 
\begin{equation*}
\lim_{B\rightarrow 1}\delta \left( B\right) =0,
\end{equation*}
cf. \cite{gm2}.
\begin{remark}
As well-known in the literature, standard needlets describe a tight frame with tightness
constant equal to $1$, allowing for an exact reconstruction formula (cf. 
\cite{npw1,npw2} and the textbook \cite[Section 10.3]{marpecbook}). On the other
hand, Mexican needlets are characterized by a non-compact support in the harmonic domain,
and this makes perfect reconstruction unfeasible for the lack of an exact
cubature formula. Despite these features, Mexican needlets enjoy some
remarkable advantages with respect to the standard ones. More specifically, they
are characterized by an extremely good concentration properties in the real domain. In addition,
it is possible to choose the measurable disjoint sets $E_{jk}$ with minimal
conditions, and still ensure frame constants arbitrarily close to unity (and
hence almost exact reconstruction).
\end{remark}
Note that, in this paper, we investigate the exact
dependence of localization properties upon $s,$ an issue which is extremely
relevant for applications (see, for example, \cite{scodeller}). It may be
noted that the choice of $s$ represents a trade-off between localization in
real and harmonic domain; the latter improves as $s$ increases, while the
reverse holds for the former.\\
\noindent We add here the following result, which establishes a link between Mexican needlets with different shape parameter $s\in\mathbb{N}$.
\begin{proposition}
	\label{delta_s} For any $s\in \mathbb{N}$, where $s>1$, and $x\in \mathbb{S}^{2}$, 
	\begin{equation*}
	\psi _{jk;s}\left( x\right) =\left( -1\right) ^{s}B^{-2js}\left( \Delta _{%
		\mathbb{S}^{2}}\right) ^{s}\psi _{jk;1}\left( x\right) .
	\end{equation*}
\end{proposition}
\begin{proof}
	Easy calculations lead to%
	\begin{align*}
	&-B^{-2j}\Delta _{\mathbb{S}^{2}}\psi _{jk;s}\left( x\right) \\
		& \quad\quad\quad\quad\quad =-\Delta _{\mathbb{S}^{2}}\left( \frac{\sqrt{\lambda _{jk}}}{B^{2j}}%
		\sum_{\ell\geq 0}\left( \frac{-e_{\ell}}{B^{2j}}\right) ^{s}\exp \left( \frac{e_{\ell}%
		}{B^{2j}}\right) \sum_{m=-\ell}^{\ell}\overline{Y}_{\ell m}\left( \xi _{jk}\right)
		Y_{\ell m}\left( x\right) \right) \\
	& \quad\quad\quad\quad\quad =\sqrt{\lambda _{jk}}\sum_{\ell\geq 0}\left( \frac{-e_{\ell}}{B^{2j}}\right)
		^{s+1}\exp \left( \frac{-e_{\ell}}{B^{2j}}\right) \sum_{m=-\ell}^{\ell}\overline{Y}%
		_{\ell m}\left( \xi _{jk}\right) Y_{\ell m}\left( x\right)\\
	& \quad\quad\quad\quad\quad =\psi _{jk;s+1}\left( x\right) .
	\end{align*}%
	Iterating the procedure, we obtain the statement.
\end{proof}

\noindent Before concluding this Section, for $\xi _{jk},x\in \mathbb{S}^{2}$, let us label the geodesic distance 
by  
\begin{equation*}\label{geodesic}
\vartheta :=\vartheta _{jk}\left( x\right) =\langle x,\xi
_{jk}\rangle ,
\end{equation*}%
so that we can express the Mexican needlets given by (\ref{mex1}) in terms
of $\vartheta $ 
\begin{equation}
\Psi _{jk;s}\left( \vartheta \right) :=\sqrt{\lambda _{jk}}\frac{1}{2\pi }%
\sum_{\ell=0}^{\infty }f_{s}\left( \frac{\left( \ell+\frac{1}{2}\right) }{B^{j}}%
\right) \left(\ell+\frac{1}{2}\right) P_{\ell}\left( \cos \vartheta \right) .\label{mexican}
\end{equation}

\section{The localization property\label{sec:loc}}
The aim of this Section is to achieve an exhaustive proof of the so-called
localization property, i.e., to establish an upper bound for the supremum of
the modulus of the Mexican needlet defined by (\ref{mexican}), remarking its
dependence on the resolution level $j$ and on the shape parameter $s$, up to
a multiplicative constant. This result is given in the Theorem \ref{theo:loc}%
. We stress again that this achievement was pursued implicitely by Geller
and Mayeli in \cite{gm2}, where the authors anyway found a similar result
studying (\ref{mexican}) for small and large angles, even if they limited
their investigations to the case $s=1$. Here, instead, we extend this
result to any value of the shape parameter $s$ in (\ref{weight}), holding for
any value of $\vartheta $ by means of a unique procedure, which resembles the
one employed by Narcowich, Petrushev and Ward in \cite{npw1} to exploit the
localization property for standard needlets on the $n$-dimensional sphere $%
\mathbb{S}^{n}$ (see also \cite{npw2} and the textbook \cite[Section 13.3]{marpecbook}). In
this case, howsoever, we will take advantage of the explicit formulation of
the weight function \eqref{weight}, which allows us to compute exactly its
Fourier transform in terms of Hermite polynomials and, through that, to
exploit precisely the dependence on the resolution level $j$ of the $\sup
\left\vert \Psi _{jk;s}\left( \vartheta \right) \right\vert $. For the sake
of simplicity, let us introduce the following notation
\begin{equation*}
\varepsilon =\varepsilon \left( B,j\right) :=B^{-j}\text{ ,}
\end{equation*}%
so that, we can define
\begin{equation}
\Psi _{\varepsilon ;s}\left( \vartheta \right) :=\frac{1}{2\pi }%
\sum_{\ell=0}^{\infty }f_{s}\left( \varepsilon \left( \ell+\frac{1}{2}\right)
\right) \left( \ell+\frac{1}{2}\right) P_{\ell}\left( \cos \vartheta \right). \label{mexican2}
\end{equation}

\begin{remark}
\label{relation}Observe that%
\begin{equation*}
\Psi _{jk;s}\left( x\right) =\Psi _{jk;s}\left( \vartheta\right) =\sqrt{\lambda _{jk}}\Psi _{\varepsilon ;s}\left( \vartheta
 \right) ,
\end{equation*}%
 Furthermore, in \eqref{mexican2}, while the index $\varepsilon $ substitutes $j$, the index $k$ is no more necessary. 
As stressed above, $\lambda _{jk}$ does not appear in \eqref{mexican2}, while Theorem \ref%
{theo:loc} holds for any $\vartheta \in \left[ 0,\pi \right] $. The
dependence on $k$ arises when we choose $\vartheta =\vartheta _{jk}\left(
x\right) $.
\end{remark}

\begin{theorem}
\label{theo:loc}Let $\Psi _{jk;s}\left( \vartheta \right) $ be
given by \eqref{mex1}. Then, for any $s\in \mathbb{N}$ and $k=1,\ldots,N_{j}$, there
exists $C_{s}>0$ such that 
\begin{equation*}
\left\vert \Psi _{jk;s}\left( x\right) \right\vert \leq C_{s}B^{j}e^{-\frac{%
		B^{2j}}{4}\vartheta ^{2}\left( x\right) }\left( 1+\left\vert B^{j}\vartheta
\left( x\right) \right\vert ^{2s}\right),
\end{equation*}
uniformly over $j$.
\end{theorem}

\begin{proof}
By using the Mehler-Dirichlet representation formula (see, for instance, \cite[Formula 4.8.7, pag. 86]{szego}), the Legendre polynomial of degree $\ell$ can
be written as 
\begin{equation*}
P_{\ell}\left( \cos \vartheta \right) =\frac{\sqrt{2}}{\pi }\int_{\vartheta
}^{\pi }\frac{\sin \left( \left( \ell+\frac{1}{2}\right) \phi \right) }{\sqrt{%
\cos \vartheta -\cos \phi }}d\phi.
\end{equation*}%
Hence, using \eqref{mexican2}, we have that%
\begin{eqnarray*}
\left\vert \Psi _{\varepsilon ;s}\left( \vartheta \right) \right\vert
&=&\left\vert \frac{1}{2\pi }\sum_{\ell=0}^{\infty }f_{s}\left( \varepsilon
\left( \ell+\frac{1}{2}\right) \right) \left( \ell+\frac{1}{2}\right)
\int_{\vartheta }^{\pi }\frac{\sin \left( \left( \ell+\frac{1}{2}\right) \phi
\right) }{\sqrt{\cos \vartheta -\cos \phi }}d\phi \right\vert \\
&=&\frac{1}{2\pi }\int_{\vartheta }^{\pi }\frac{\left\vert
\sum_{\ell=0}^{\infty }f_{s}\left( \varepsilon \left( \ell+\frac{1}{2}\right)
\right) \left(\ell+\frac{1}{2}\right) \sin \left( \left(\ell+\frac{1}{2}\right)
\phi \right) \right\vert }{\sqrt{\cos \vartheta -\cos \phi }}d\phi \\
&\leq &\frac{1}{2\pi }\int_{\vartheta }^{\pi }\frac{\Lambda _{\varepsilon
,s}\left( \phi \right) }{\sqrt{\cos \vartheta -\cos \phi }}d\phi \text{ ,}
\end{eqnarray*}%
where%
\begin{eqnarray}
\Lambda _{\varepsilon ;s}\left( \phi \right) &:=&\sum_{\ell=0}^{\infty
}f_{s}\left( \varepsilon \left( \ell+\frac{1}{2}\right) \right) \left(\ell+\frac{1%
}{2}\right) \sin \left( \left( \ell+\frac{1}{2}\right) \phi \right)  \notag \\
&=&\sum_{\ell=0}^{\infty }g_{\varepsilon ,\phi ;s}\left( \ell+\frac{1}{2}\right)\notag \\
&=&\frac{1}{2}\sum_{\ell=-\infty }^{\infty }g_{\varepsilon ,\phi ;s}\left( \ell+%
\frac{1}{2}\right). \label{Kappa}
\end{eqnarray}%
In the last equality, we use the fact that%
\begin{equation*}
g_{\varepsilon ,\phi ;s}\left( u\right) :=f_{s}\left( \varepsilon u\right)
u\sin \left( u\phi \right) \text{ , }u\in \mathbb{R}
\end{equation*}%
is an even function.
Using Lemma \ref{lemma_kappa}, we obtain %
\begin{equation}
\left\vert \Psi _{\varepsilon ;s}\left( \vartheta \right) \right\vert \leq 
\frac{\widetilde{C}_{2s+1}}{\varepsilon ^{2}}\left\vert \int_{\vartheta
}^{\pi }\frac{e^{-\left( \frac{\phi }{2\varepsilon }\right)
^{2}}H_{2s+1}\left( \frac{\phi }{2\varepsilon }\right) }{\sqrt{\cos
\vartheta -\cos \phi }}d\phi \right\vert. \label{intlast}
\end{equation}%
Note that 
\begin{equation}
\cos \vartheta -\cos \phi =2\left( \phi ^{2}-\vartheta ^{2}\right) \frac{%
\sin \left( \frac{\vartheta +\phi }{2}\right) \sin \left( \frac{\phi
-\vartheta }{2}\right) }{\left( \frac{\vartheta +\phi }{2}\right) \left( 
\frac{\phi -\vartheta }{2}\right) }\text{ .}  \label{approx_cos}
\end{equation}%
In order to estimate \eqref{intlast}, we consider three different cases. In
case I, we have that $\vartheta \in \left( \delta ,\frac{\pi }{2}\right) $, where $%
0<\delta <\varepsilon $. In case II, $\vartheta \in \left[ \frac{\pi }{2}%
,\pi \right] $. Finally, in case III, we have that $\vartheta \in \left[ 0,\delta \right] $%
, $0<\delta <\varepsilon $, as in \cite{npw1}. In cases I and II, we will prove
that there exists a constant $\widetilde{C}_{s}$ so that%
\begin{equation*}
\left\vert \Psi _{\varepsilon ;s}\left( \vartheta \right) \right\vert \leq 
\frac{\widetilde{C}_{s}}{\varepsilon ^{2}}e^{-\left( \frac{\vartheta }{%
2\varepsilon }\right) ^{2}}\left\vert H_{2s}\left( \frac{\vartheta }{%
2\varepsilon }\right) \right\vert . 
\end{equation*}%
We distinguish between the two cases for technical reasons. Indeed, in case II, the integral in \eqref%
{intlast} will be estimated by using supplementary angles. As far as case III is concerned, we
will prove that there exists $C_{s}^{\prime \prime \prime }$ such that 
\begin{equation*}
\left\vert \Psi _{\varepsilon ;s}\left( \vartheta \right) \right\vert \leq 
\frac{C_{s}^{\prime \prime \prime }}{\varepsilon ^{2}}.
\end{equation*}
\noindent \textit{Case I.} We have that%
\begin{eqnarray*}
0 &<&\frac{\vartheta +\phi }{2}\leq \frac{3}{4}\pi \text{ ,} \\
0 &\leq &\frac{\phi -\vartheta }{2}\leq \frac{\pi }{2}\text{ .}
\end{eqnarray*}%
On one hand, note that \eqref{approx_cos} can be bounded as:%
\begin{eqnarray*}
\cos \vartheta -\cos \phi &\geq &\frac{1}{2}\left( \phi ^{2}-\vartheta
^{2}\right) \frac{\sqrt{2}}{2}\frac{4}{3\pi }\frac{\sqrt{2}}{2}\frac{4}{\pi }
\\
&=&C_{I}\left( \phi ^{2}-\vartheta ^{2}\right), 
\end{eqnarray*}%
where $C_{I}>0$. On the other hand, the integral \eqref{intlast} can be rewritten as%
\begin{equation*}
\left\vert \Psi _{\varepsilon ;s}\left( \vartheta \right) \right\vert \leq 
\frac{\widetilde{C}_{2s+1}^{\prime }}{\varepsilon ^{2}}\left\vert
\int_{\vartheta }^{\pi }\frac{e^{-\left( \frac{\phi }{2\varepsilon }\right)
^{2}}H_{2s+1}\left( \frac{\phi }{2\varepsilon }\right) }{\sqrt{\left( \phi
^{2}-\vartheta ^{2}\right) }}d\phi \right\vert  ,
\end{equation*}%
where $\widetilde{C}_{2s+1}^{\prime }>0$.
Recall that, for $n$ odd and $u\in \mathbb{R}$,
the Hermite polynomials can be rewritten as 
\begin{equation}
H_{n}\left( u\right) =n!\sum_{r=0}^{\frac{n-1}{2}}\frac{\left( -1\right) ^{%
\frac{n-1}{2}-r}}{\left( 2r+1\right) !\left( \frac{n-1}{2}-r\right) !}\left(
2u\right) ^{2r+1}, \label{hermiteodd}
\end{equation}%
(see, for instance, \cite[Chapter 22]{abramostegun}). 
Therefore, we have that
\begin{equation*}
\left\vert \Psi _{\varepsilon ;s}\left( \vartheta \right) \right\vert \leq 
\frac{\widetilde{C}_{2s+1}^{\prime }}{\varepsilon ^{2}}\left( 2s+1\right)
!\left\vert \sum_{r=0}^{s}\frac{\left( -1\right) ^{s-r}2^{2r+1}}{\left(
2r+1\right) !\left( s-r\right) !}Q_{\varepsilon ,r}\left( \vartheta \right) 
\text{ }\right\vert ,
\end{equation*}%
where%
\begin{equation*}
Q_{\varepsilon ,r}\left( \vartheta \right) :=\int_{\vartheta }^{\pi }\frac{%
e^{-\left( \frac{\phi }{2\varepsilon }\right) ^{2}}\left( \frac{\phi }{%
2\varepsilon }\right) ^{2r+1}}{\sqrt{\left( \phi ^{2}-\vartheta ^{2}\right) }%
}d\phi.
\end{equation*}%
Using Lemma \ref{prop:qu}, which establishes that 
\begin{equation*}
Q_{\varepsilon ,r}\left( \vartheta \right) \leq C_{r}e^{-\left( \frac{%
\vartheta }{2\varepsilon }\right) ^{2}}\left( \frac{\vartheta }{2\varepsilon 
}\right) ^{2r},
\end{equation*}
where $C_{r}>0$, we get
\begin{eqnarray*}
\left\vert \Psi _{\varepsilon ;s}\left( \vartheta \right) \right\vert &\leq &%
\frac{\widetilde{C}_{2s+1}^{\prime }}{\varepsilon ^{2}}\left\vert \left(
2s+1\right) !\sum_{r=0}^{s}\frac{\left( -1\right) ^{s-r}2^{2r+1}}{\left(
2r+1\right) !\left( s-r\right) !}Q_{\varepsilon ,r}\left( \vartheta \right)
\right\vert \\
&\leq &\frac{\widetilde{C}_{2s+1}^{\prime }}{\varepsilon ^{2}}e^{-\left( 
\frac{\vartheta }{2\varepsilon }\right) ^{2}}\left\vert \left( 2s\right)
!\sum_{r=0}^{s}C_{r}\frac{\left( -1\right) ^{s-r}2^{2r+1}}{\left( 2r\right)
!\left( s-r\right) !}\left( \frac{\vartheta }{2\varepsilon }\right)
^{2r}\right\vert \\
&=&\frac{C_{s}^{\prime }}{\varepsilon ^{2}}e^{-\left( \frac{\vartheta }{%
2\varepsilon }\right) ^{2}}\left\vert H_{2s}\left( \frac{\vartheta }{%
2\varepsilon }\right) \right\vert, 
\end{eqnarray*}
where $C_{s}^{\prime }>0$.\\

\noindent \textit{Case II.} As in \cite{npw1} (see also \cite[Section
13.3]{marpecbook}, we use the supplementary angles to $\phi$ and $\vartheta$, denoted by $\widetilde{\phi } =\pi -\phi$ and 
$\widetilde{\vartheta } =\pi -\vartheta$, respectively. We can easily observe that%
\begin{eqnarray*}
0 &\leq &\frac{\widetilde{\vartheta }+\widetilde{\phi }}{2}\leq \frac{\pi }{2%
}; \\
0 &\leq &\frac{\widetilde{\vartheta }-\widetilde{\phi }}{2}\leq \frac{\pi }{2%
},
\end{eqnarray*}%
so that we get%
\begin{eqnarray*}
\cos \widetilde{\phi }-\cos \widetilde{\vartheta } &\geq &2\left( \widetilde{%
\vartheta }^{2}-\widetilde{\phi }^{2}\right) \frac{\sin \left( \frac{%
\widetilde{\vartheta }-\widetilde{\phi }}{2}\right) \sin \left( \frac{%
\widetilde{\vartheta }+\widetilde{\phi }}{2}\right) }{\left( \frac{%
\widetilde{\vartheta }-\widetilde{\phi }}{2}\right) \left( \frac{\widetilde{%
\vartheta }+\widetilde{\phi }}{2}\right) } \\
&\geq &C_{II}\left( \widetilde{\vartheta }^{2}-\widetilde{\phi }^{2}\right), 
\end{eqnarray*}%
where $C_{II}>0$.
By substitution in \eqref{intlast}, following the same procedure
as above yields
\begin{eqnarray*}
\left\vert \Psi _{\varepsilon ;s}\left( \vartheta \right) \right\vert &\leq &%
\frac{\widetilde{C}_{2s+1}^{\prime \prime }}{\varepsilon ^{2}}\left\vert
\int_{0}^{\widetilde{\vartheta }}\frac{e^{-\left( \frac{\pi -\widetilde{\phi 
}}{2\varepsilon }\right) ^{2}}H_{2s+1}\left( \frac{\pi -\widetilde{\phi }}{%
2\varepsilon }\right) }{\sqrt{\left( \widetilde{\vartheta }^{2}-\widetilde{%
\phi }^{2}\right) }}d\widetilde{\phi }\right\vert \\
&\leq &\frac{\widetilde{C}_{2s+1}^{\prime \prime }}{\varepsilon ^{2}}%
\left\vert \left( 2s+1\right) !\sum_{r=0}^{s}\frac{\left( -1\right) ^{s-r}}{%
\left( 2r+1\right) !\left( s-r\right) !}\widetilde{Q}_{\varepsilon ,r}\left(
\vartheta \right) \right\vert ,
\end{eqnarray*}
where $\widetilde{C}_{2s+1}^{\prime \prime
}>0$ and 
\begin{equation*}
\widetilde{Q}_{\varepsilon ,r}\left( \vartheta \right) :=\int_{0}^{%
\widetilde{\vartheta }}\frac{e^{-\left( \frac{\widetilde{\phi }}{%
2\varepsilon }\right) ^{2}}\left( \frac{\pi -\widetilde{\phi }}{2\varepsilon 
}\right) ^{2r+1}}{\sqrt{\left( \widetilde{\vartheta }^{2}-\widetilde{\phi }%
^{2}\right) }}d\widetilde{\phi }.
\end{equation*}%
Finally, using Lemma \ref{prop:qu} leads to
\begin{equation*}
\widetilde{Q}_{\varepsilon ,r}\left( \vartheta \right) \leq \widetilde{C}%
_{r}\left( \frac{\widetilde{\vartheta }}{2\varepsilon }\right)
^{2r}e^{-\left( \frac{\widetilde{\vartheta }}{2\varepsilon }\right) ^{2}},
\end{equation*}%
and, as a straightforward consequence, we get%
\begin{equation*}
\left\vert \Psi _{\varepsilon ;s}\left( \vartheta \right) \right\vert \leq 
\frac{C_{s}^{\prime \prime }}{\varepsilon ^{2}}e^{-\left( \frac{\vartheta }{%
2\varepsilon }\right) ^{2}}\left\vert H_{2s}\left( \frac{\vartheta }{%
2\varepsilon }\right) \right\vert \text{ .}
\end{equation*}%
\noindent \textit{Case III. }Following again \cite{npw1}, we have that
\begin{eqnarray*}
\Lambda _{\varepsilon ;s}\left( \phi \right) &\leq &\frac{1}{\varepsilon ^{2}%
}\sum_{\ell\geq 0}f_{s}\left( \varepsilon \left( \ell+\frac{1}{2}\right) \right)
\varepsilon ^{2}\left( \ell+\frac{1}{2}\right) \\
&\leq &\frac{C_{III}}{\varepsilon ^{2}}\sum_{\ell\geq 0}\left( \varepsilon
\ell\right) ^{2s+1}e^{-\left( \varepsilon \ell\right) ^{2}}\int_{\varepsilon
\ell}^{e\left( \ell+1\right) }du \\
&\leq &\frac{C_{III}}{\varepsilon ^{2}}\int_{0}^{\infty }u^{2s+1}e^{-u^{2}}du
\\
&=&\frac{C_{III}}{\varepsilon ^{2}}\frac{\Gamma \left( s+\frac{3}{2}\right) 
}{2}=\frac{C_{s}^{\prime \prime \prime }}{\varepsilon ^{2}}\text{ .}
\end{eqnarray*}%
Combining these results yields 
\begin{equation*}
\left\vert \Psi _{\varepsilon ;s}\left( \vartheta \right) \right\vert <C_{s}%
\frac{e^{-\left( \frac{\vartheta }{2\varepsilon }\right) ^{2}}}{\varepsilon
^{2}}\left( 1+\left\vert H_{2s}\left( \frac{\vartheta }{\varepsilon }\right)
\right\vert \right),
\end{equation*}%
for $\vartheta =\left[ 0,\pi \right]$. From Remark \ref{relation} and since $\lambda _{jk}\leq cB^{-2j}$ , we
have that
\begin{equation*}
\left\vert \Psi _{jk;s}\left( x\right) \right\vert \leq C_{s}B^{j}e^{-\frac{%
B^{2j}}{4}\vartheta ^{2}\left( x\right) }\left( 1+\left\vert B^{j}\vartheta
\left( x\right) \right\vert ^{2s}\right),
\end{equation*}
as claimed.
\end{proof}
\begin{remark}
In view of Remark \ref{remark:equivalence}, it holds
\begin{equation*}
\left\vert \psi _{jk;s}\left( x\right) \right\vert \leq C_{s}B^{j}e^{-\frac{%
		B^{2j}}{4}\vartheta ^{2}\left( x\right) }\left( 1+\left\vert B^{j}\vartheta
\left( x\right) \right\vert ^{2s}\right).
\end{equation*}
\end{remark}
\begin{remark}\label{scodellerlink}
		As suggested in \cite{scodeller}, Mexican needlets in the case $s=1$ provide a valid asymptotic approximation to the spherical Mexican hat wavelets . Recall that the discretized version of spherical Mexican hat wavelets use a stereographic projection on the sphere (see \cite{ant}). These wavelets conserve the most crucial properties of the flat Mexican hat wavelets and, for this reason, they are widely used in Astrophysics and Cosmology, even if they lack of a reconstruction formula (see, for example, \cite{mcewen}, for a quick review). In \cite{scodeller}, it is proved that the bound between the absolute value of the difference between spherical Mexican hat wavelets and Mexican needlets is of order $B^{-j} \min\left( \vartheta^4B^{4j},1\right)$. This bound matches exactly with the results proved in Theorem \ref{theo:loc}. Indeed, fixed $s=1$, the bound for small angles is controlled by $B^{-j}$, which depends, on one hand, on the normalization factor of the spherical Mexican hat wavelet and, on the other, on $\sqrt{\lambda_{jk}}$. Up to a proper normalization, for larger angles, this factor has to be multiplied by a series expansion of even powers of $\vartheta$, controlled by leading term of order $4$ (cf. \cite{gm1}). Heuristically, it implies that spherical Mexican hat wavelets can be approximated to Mexican needlets in the corresponding $E_{jk}$ and that this approximation is better for large $j$.       
		For this reason, the spatial concentration properties here discussed and the correlation properties studied in \cite{lanmar2,mayeli} can be helpful for the study of the asymptotic behaviour of the random spherical Mexican wavelet coefficients hat in the high-frequency limit.          
\end{remark}
Before concluding this Section, as for the standard needlets (see \cite%
{npw1,npw2}), we can also establish the order of the $L^{p}$-norms of
Mexican needlets as follows.
\begin{corollary}[Bounds on $L^{p}\left( \mathbb{S}^{2}\right) $-norms]\label{coronorms}
For any $p\in \left[ 1,\infty \right) \,$, there exist $c_{p},C_{p}\in 
\mathbb{R}$ such that%
\begin{equation*}
c_{p}B^{2j\left( \frac{1}{2}-\frac{1}{p}\right) }\leq \left\Vert \Psi
_{jk;s}\right\Vert _{L^{p}\left( \mathbb{S}^{2}\right) }\leq
C_{p}B^{2j\left( \frac{1}{2}-\frac{1}{p}\right) }.
\end{equation*}%
Furthermore, there exist $c_{\infty },C_{\infty }\in \mathbb{R}$ such that 
\begin{equation*}
c_{\infty }B^{j}\leq \left\Vert \Psi _{jk;s}\right\Vert _{L^{\infty }\left( 
\mathbb{S}^{2}\right) }\leq C_{\infty }B^{j}.
\end{equation*}
\end{corollary}

\begin{proof}
The proof of this Corollary is very close to the one developed in the
standard needlet framework in \cite{npw2}. The only remarkable difference
concerns the estimate of the bounds for $L^{2}\left( \mathbb{S}^{2}\right) 
$ norms. In \cite{npw2}, this bound is proven as corollary of the
tight-frame property. We establish a similar result for the Mexican needlet
framework as follows. Let $dx$ denote the uniform spherical measure. Hence, we have that
\begin{align*}
 \left\Vert \Psi _{jk;s}\right\Vert _{L^{2}\left( \mathbb{S}^{2}\right)
}^{2}=&\int_{\mathbb{S}^{2}}\left\vert \Psi _{jk;s}\left( x\right)
\right\vert ^{2}dx \\ 
=&\lambda _{jk}\int_{\mathbb{S}^{2}}\sum_{\ell=0}^{\infty }\sum_{\ell^{\prime
}=0}^{\infty }f_{s}\left( \frac{\ell}{B^{j}}\right) f_{s}\left( \frac{\ell^{\prime
}}{B^{j}}\right) \\
&\times \sum_{m=-\ell}^{\ell}\sum_{m^{\prime }=-\ell^{\prime}}^{\ell^{\prime}}Y_{\ell m}\left( x\right) 
\overline{Y}_{\ell^{\prime }m^{\prime }}\left( x\right) \overline{Y}_{\ell m}\left(
\xi _{jk}\right) Y_{\ell^{\prime }m^{\prime }}\left( \xi _{jk}\right) dx \\
=&\lambda _{jk}\sum_{\ell=0}^{\infty }f_{s}^{2}\left( \frac{\ell}{B^{j}}\right)
\sum_{m=-\ell}^{\ell}\overline{Y}_{\ell m}\left( \xi _{jk}\right) Y_{\ell m}\left( \xi
_{jk}\right) \delta _{\ell}^{\ell^{\prime }}\delta _{m}^{m^{\prime }} \\
 =&\lambda _{jk}\sum_{\ell=0}^{\infty }f_{s}^{2}\left( \frac{\ell}{B^{j}}\right) 
\frac{2\ell+1}{4\pi }.
\end{align*}%
On one hand, we get
\begin{align*}
\lambda _{jk}\sum_{\ell=0}^{\infty }f_{s}^{2}\left( \frac{\ell}{B^{j}}\right) 
\frac{2\ell+1}{4\pi } \leq &\frac{c}{2\pi }\frac{1}{B^{j}}\sum_{\ell=0}^{\infty }f_{s}^{2}\left( 
\frac{\ell}{B^{j}}\right) \frac{\ell+1/2}{B^{j}} \\
&=\frac{c}{2\pi }\frac{1}{B^{j}}\sum_{\ell=0}^{\infty }\left( \frac{\ell}{B^{j}}%
\right) ^{4s} e^{  -2\left( \frac{\ell}{B^{j}}\right) ^{2}} \frac{%
\ell+1/2}{B^{j}} \\
&\leq \frac{c}{2\pi }\sum_{\ell=0}^{\infty }\left( \frac{\ell}{B^{j}}\right)
^{4s} e^{ -2\left( \frac{\ell}{B^{j}}\right) ^{2} }\frac{\ell}{B^{j}}%
\int_{\frac{\ell}{B^{j}}}^{\frac{\ell+1}{B^{j}}}du \\
&\leq\frac{c}{2\pi }\int_{0}^{\infty }u^{4s+1} e^{ -2u^{2}}
du\leq C_{2}.
\end{align*}%
On the other hand, we obtain 
\begin{equation*}
\lambda _{jk}\sum_{\ell=0}^{\infty }f_{s}^{2}\left( \frac{\ell}{B^{j}}\right) 
\frac{2\ell+1}{4\pi }\geq c_{2}.
\end{equation*}%
Following \cite{npw2}, we get 
\begin{equation*}
c_{p}B^{2j\left( \frac{1}{2}-\frac{1}{p}\right) }\leq \left\Vert \Psi
_{jk;s}\right\Vert _{L^{p}\left( \mathbb{S}^{2}\right) }\leq
C_{p}B^{2j\left( \frac{1}{2}-\frac{1}{p}\right) },
\end{equation*}
as claimed.
\end{proof}

\section{Auxiliary results\label{sec:aux}}

In this Section we collect some auxiliary results, concerning the
upper bounds of $\Lambda _{\varepsilon ;s}$, $Q_{\varepsilon ,r}$ and $%
\widetilde{Q}_{\varepsilon ,r}$.\\We introduce preliminarily the following notation for the Fourier transform of a function $f\in
L^{1}\left( \mathbb{R}\right) $: 
\begin{equation*}
F\left[ f\right] \left( \omega \right) :=\int_{\mathbb{R}}f\left( u\right)
e^{-i\omega u}du=:\widehat{f}\left( \omega \right) \text{ .}
\end{equation*}%
Let us also recall two standard properties for the Fourier transforms. Under
standard conditions, we have that
\begin{align*}
&\frac{d^{\alpha }}{d\omega ^{\alpha }}\widehat{f}\left( \omega \right)
=\left( -i\right) ^{\alpha }F\left[ u^{\alpha }f\left( u\right) \right]
\left( \omega \right); \label{prop_der1} \\
&F\left[ \frac{d^{\alpha }}{du^{\alpha }}f\left( u\right) \right] \left(
\omega \right) =\left( -i\right) ^{\alpha }\omega ^{\alpha }F\left[
f\left( u\right) \right] \left( \omega \right).  
\end{align*}%
Finally, the Poisson Summation Formula can be defined as follows. If, for $\omega \in %
\left[ 0,2\pi \right] $ and $\alpha >0$, 
\begin{equation*}
\left\vert f\left( u\right) \right\vert +\left\vert \widehat{f}\left( \omega
\right) \right\vert \leq \frac{C_{a}}{1+\left\vert u\right\vert ^{\alpha +1}},
\end{equation*}%
then:%
\begin{equation}
\sum_{\tau =-\infty }^{\infty }f\left( \tau \right) e^{-i\omega \tau
}=\sum_{\nu =-\infty }^{\infty }\widehat{f}\left( \omega +2\pi \nu \right) 
\text{ .}  \label{prop_psf1}
\end{equation}%
For more details and discussions about Fourier transforms, the reader is referred, for
instance, to the textbook \cite{steinweiss}.

\begin{lemma}
\label{lemma_kappa}Let $\Lambda _{\varepsilon ;s}\left( \phi \right) $ be
given by (\ref{Kappa}). Then there exists $\widetilde{C}_{2s+1}>0$ such
that 
\begin{equation*}
\Lambda _{\varepsilon ;s}\left( \phi \right) \leq \frac{\widetilde{C}_{2s+1}%
}{\varepsilon ^{2}}e^{-\left( \frac{\phi }{2\varepsilon }\right)
^{2}}\left\vert H_{2s+1}\left( \frac{\phi }{2\varepsilon }\right)
\right\vert \text{ .}
\end{equation*}
\end{lemma}

\begin{proof} First, note that straightforward calculations lead to
 \begin{equation*}
F\left[ g_{\varepsilon ,\phi ;s}\left( u+\frac{1}{2}\right) \right] \left(
\omega \right) =e^{i\frac{\omega }{2}}\widehat{g}_{\varepsilon ,\phi
;s}\left( \omega \right) .
\end{equation*}%
On one hand, we have that
\begin{eqnarray*}
F\left[ f_{s}\left( \varepsilon u\right) u\right] \left( \omega \right)
&=&\int_{\mathbb{R}}f_{s}\left( \varepsilon u\right) ue^{-i\omega u} \\
&=&\frac{1}{\varepsilon ^{2}}F\left[ f_{s}\left( u\right) u\right] \left( 
\frac{\omega }{\varepsilon }\right).
\end{eqnarray*}%
On the other hand, note that%
\begin{equation}\label{gausstransf}
F\left[ e^{-u^{2}}\right] \left( \omega \right) =\sqrt{\pi }e^{-\frac{\omega
^{2}}{4}}.
\end{equation}%
Combining \eqref{weight} and \eqref{gausstransf} yields  
\begin{eqnarray*}
F\left[ f_{s}\left( u\right) u\right] \left( \omega \right) &=&i^{2s+1}\frac{%
d^{2s+1}}{d\omega ^{2s+1}}F\left[ e^{-u^{2}}\right] \left( \omega \right) \\
&=&i^{2s+1}\sqrt{\pi }\frac{d^{2s+1}}{d\omega ^{2s+1}}e^{-\frac{\omega ^{2}}{%
4}} \\
&=&\left( -1\right) ^{\left( s+\frac{1}{2}\right) }\sqrt{\pi }H_{2s+1}\left( 
\frac{\omega }{2}\right) e^{-\frac{\omega ^{2}}{4}}\text{ ,}
\end{eqnarray*}%
where $H_{2s+1}\left( \cdot \right) $ is the Hermite polynomial of order $%
2s+1$. Recall that the polynomials composing $H_{n}\left( \cdot \right) $
are all even (odd) if $n$ is even (odd) - for more details, see, for instance, 
\cite[Chapter 22]{abramostegun}. Collecting all these results, we get%
\begin{equation*}
F\left[ f_{s}\left( \varepsilon u\right) u\right] \left( \omega \right) =%
\frac{\left( -1\right) ^{\left( s+\frac{1}{2}\right) }\sqrt{\pi }}{%
\varepsilon ^{2}}H_{2s+1}\left( \frac{\omega }{2\varepsilon }\right)
e^{-\left( \frac{\omega }{2\varepsilon }\right) ^{2}}.
\label{trans2}
\end{equation*}%
Hence, we obtain:%
\begin{align*}
\widehat{g}_{\varepsilon ,\phi ;s}\left( \omega \right) =&F\left[ \sin
\left( \phi u\right) \right] \left( \omega \right) \ast F\left[ f_{s}\left(
\varepsilon u\right) u\right] \left( \omega \right) \\
=&\frac{\left( -1\right) ^{s}\pi ^{\frac{3}{2}}}{\varepsilon ^{2}}\left(\!
H_{2s+1}\left( \frac{\omega -\phi }{2\varepsilon }\right) e^{-\left( \frac{%
\omega -\phi }{2\varepsilon }\right) ^{2}} \!-\!H_{2s+1}\left( \frac{\omega +\phi 
}{2\varepsilon }\right) e^{-\left( \frac{\omega +\phi }{2\varepsilon }%
\right) ^{2}}\right).
\end{align*}%
Using (\ref{prop_psf1}), we have that%
\begin{align*}
\sum_{\ell=-\infty }^{\infty }g_{\varepsilon ,\phi ;s}\left( \ell+\frac{1}{2}%
\right) =&\sum_{\nu =-\infty }^{\infty }e^{i\frac{2\pi \nu }{2}}\widehat{g}%
_{\varepsilon ,\phi ;s}\left( 2\pi \nu \right)\\
=&\sum_{\nu =-\infty }^{\infty }e^{i\frac{2\pi \nu }{2}}\frac{\left(
-1\right) ^{s}\pi ^{\frac{3}{2}}}{\varepsilon ^{2}}\left( H_{2s+1}\left( 
\frac{2\pi \nu -\phi }{2\varepsilon }\right) e^{-\left( \frac{2\pi \nu -\phi 
}{2\varepsilon }\right) ^{2}}\right. \\
&-\left. H_{2s+1}\left( \frac{2\pi \nu +\phi }{2\varepsilon }\right)
e^{-\left( \frac{2\pi \nu +\phi }{2\varepsilon }\right) ^{2}}\right) \\
=&2\sum_{\nu =-\infty }^{\infty }e^{i\frac{2\pi \nu }{2}}\frac{\left(
-1\right) ^{s+1}\pi ^{\frac{3}{2}}}{\varepsilon ^{2}}H_{2s+1}\left( \frac{%
2\pi \nu +\phi }{2\varepsilon }\right) e^{-\left( \frac{2\pi \nu +\phi }{%
2\varepsilon }\right) ^{2}},
\end{align*}%
where the last equality takes into account that $H_{2s+1}\left( \cdot \right) $
is odd. Therefore, we have that 
\begin{equation*}
\Lambda _{\varepsilon ;s}\left( \phi \right) =\frac{\left( -1\right)
^{s+1}\pi ^{\frac{3}{2}}}{\varepsilon ^{2}}\sum_{\nu =-\infty }^{\infty }e^{i%
\frac{2\pi \nu }{2}}H_{2s+1}\left( \frac{2\pi \nu +\phi }{2\varepsilon }%
\right) e^{-\left( \frac{2\pi \nu +\phi }{2\varepsilon }\right) ^{2}}.
\end{equation*}%
Then, we obtain
\begin{equation*}
\left\vert \Lambda _{\varepsilon ;s}\left( \phi \right) \right\vert =\frac{%
\pi ^{\frac{3}{2}}}{\varepsilon ^{2}}\left\vert \sum_{\nu =-\infty }^{\infty
}H_{2s+1}\left( \frac{2\pi \nu +\phi }{2\varepsilon }\right) e^{-\left( 
\frac{2\pi \nu +\phi }{2\varepsilon }\right) ^{2}}\right\vert \text{ } 
\leq\frac{\pi ^{\frac{3}{2}}}{\varepsilon ^{2}}V_{\varepsilon ,s}\left(
\phi \right) ,
\end{equation*}%
where%
\begin{equation*}
V_{\varepsilon ,s}\left( \phi \right) =\sum_{\nu =-\infty }^{\infty
}\left\vert H_{2s+1}\left( \frac{2\pi \nu +\phi }{2\varepsilon }\right)
e^{-\left( \frac{2\pi \nu +\phi }{2\varepsilon }\right) ^{2}}\right\vert 
.
\end{equation*}%
Note that 
\begin{equation}
V_{\varepsilon ,s}\left( \phi \right) =\left\vert H_{2s+1}\left( \frac{\phi 
}{2\varepsilon }\right) \right\vert e^{-\left( \frac{\phi }{2\varepsilon }%
\right) ^{2}}+V_{+}+V_{-}\text{ ,}  \label{Vtot}
\end{equation}%
where%
\begin{align*}
V_{+} =&\sum_{\nu =1}^{\infty }\left\vert H_{2s+1}\left( \frac{2\pi \nu
+\phi }{2\varepsilon }\right) \right\vert e^{-\left( \frac{2\pi \nu +\phi }{%
2\varepsilon }\right) ^{2}}, \\
V_{-} =&\sum_{\nu =-1}^{-\infty }\left\vert H_{2s+1}\left( \frac{2\pi \nu
+\phi }{2\varepsilon }\right) \right\vert e^{-\left( \frac{2\pi \nu +\phi }{%
2\varepsilon }\right) ^{2}}.
\end{align*}%
Using (\ref{hermiteodd}), for $\left\vert u\right\vert >1$, we have that 
\begin{eqnarray}
\left\vert H_{n}\left( u\right) \right\vert &\leq &n!\sum_{k=0}^{\frac{n-1}{2%
}}\left\vert \frac{\left( -1\right) ^{\frac{n-1}{2}-k}}{\left( 2k+1\right)
!\left( \frac{n-1}{2}-k\right) !}\left( 2u\right) ^{2k+1}\right\vert  \notag
\\
&\leq &C_{n}^{\prime }\left\vert u\right\vert ^{n}. \label{H_lim}
\end{eqnarray}%
Hence, we get%
\begin{align*}
\left\vert H_{2s+1}\!\left( \!\frac{2\pi \nu +\phi }{2\varepsilon }\!\right)\!
\right\vert\! e^{-\left( \frac{2\pi \nu +\phi }{2\varepsilon }\right) ^{2}}\!
\leq &C_{2s+1}\left\vert \frac{2\pi \nu +\phi }{2\varepsilon }\right\vert
^{2s+1}e^{-\left( \frac{2\pi \nu +\phi }{2\varepsilon }\right) ^{2}} \\
=&C_{2s+1}\left( \frac{2\pi \nu +\phi }{2\varepsilon }\right)
^{2s+1}e^{-\left( \frac{\phi }{2\varepsilon }\right) ^{2}}e^{-\left( \frac{%
\pi \nu }{\varepsilon }\right) ^{2}}e^{-\frac{2\pi \nu \phi }{4\varepsilon
^{2}}} \\
=&e^{-\left( \frac{\phi }{2\varepsilon }\right) ^{2}}C_{2s+1}\left[ \left( 
\frac{\pi \nu }{\varepsilon }+\frac{\phi }{2\varepsilon }\right)
^{2s+1}e^{-\left( \frac{\pi \nu }{\varepsilon }\right) ^{2}}e^{-\frac{2\pi
\nu \phi }{4\varepsilon ^{2}}}\right] \\
\leq &e^{-\left( \frac{\phi }{2\varepsilon }\right) ^{2}}C_{2s+1}\left[
\left( \frac{\pi \nu }{\varepsilon }+\frac{\pi }{2\varepsilon }\right)
^{2\left( s+1\right) }e^{-\left( \frac{\pi \nu }{\varepsilon }\right) ^{2}}%
\right].
\end{align*}%
Note that%
\begin{equation*}
\frac{\pi u}{\varepsilon }+\frac{\pi }{2\varepsilon }<\frac{2\pi u}{%
\varepsilon };
\end{equation*}%
furthermore, observing that $e^{-\nu \phi }<1$, we obtain
\begin{align}
V_{+} \leq &C_{2s+1}\sum_{\nu =1}^{\infty }\left\vert H_{2\left( s+1\right)
}\left( \frac{2\pi \nu +\phi }{2\varepsilon }\right) e^{-\left( \frac{2\pi
\nu +\phi }{2\varepsilon }\right) ^{2}}\right\vert  \notag \\
\leq &e^{-\left( \frac{\phi }{2\varepsilon }\right) ^{2}}C_{2s+1}\sum_{\nu
=1}^{\infty }\left( \frac{\pi \nu }{\varepsilon }+\frac{\pi }{2\varepsilon }%
\right) ^{2s+1}e^{-\left( \frac{\pi \nu }{\varepsilon }\right) ^{2}}  \notag
\\
\leq &e^{-\left( \frac{\phi }{2\varepsilon }\right)
^{2}}C_{2s+1}2^{2s+1}\sum_{\nu =1}^{\infty }\left( \frac{\pi \nu }{%
\varepsilon }\right) ^{2s+1}e^{-\left( \frac{\pi \nu }{\varepsilon }\right)
^{2}}  \notag \\
\leq &C_{2s+1}^{\prime }e^{-\left( \frac{\phi }{2\varepsilon }\right) ^{2}}\label{Vpiu} .
\end{align}%
Indeed, the series $\sum_{\nu =1}^{\infty }\left( \frac{\pi \nu }{%
\varepsilon }\right) ^{2s+1}e^{-\left( \frac{\pi \nu }{\varepsilon }\right)
^{2}}$ is convergent, as easily proved by means of the D'Alembert's
criterion, i.e.,  
\begin{equation*}
\lim_{\nu \rightarrow \infty }\frac{\left( \frac{\pi \left( \nu +1\right) }{%
\varepsilon }\right) ^{2s+1}e^{-\left( \frac{\pi \left( \nu +1\right) }{%
\varepsilon }\right) ^{2}}}{\left( \frac{\pi \nu }{\varepsilon }\right)
^{2s+1}e^{-\left( \frac{\pi \nu }{\varepsilon }\right) ^{2}}}\\
=\lim_{\nu \rightarrow \infty }\left( 1+\frac{1}{v}\right) ^{2s+1}\exp
\left( -\frac{\pi ^{2}}{\varepsilon ^{2}}\left( 2\nu +1\right) \right) =0,
\end{equation*}%
for all $\nu >1$. On the other hand, if $\left\vert u\right\vert \leq 1$, we
obtain%
\begin{equation*}
\left\vert H_{n}\left( u\right) \right\vert \leq n!\sum_{k=0}^{\frac{n-1}{2%
}}\left\vert \frac{1}{\left( 2k+1\right) !\left( \frac{n-1}{2}-k\right) !}%
2^{2k+1}\right\vert  
\leq C_{n}^{\prime }.\label{Hlim_bis}
\end{equation*}%
Hence, we have that 
\begin{align*}
\left\vert H_{2s+1}\left( \frac{2\pi \nu +\phi }{2\varepsilon }\right)
\right\vert e^{-\left( \frac{2\pi \nu +\phi }{2\varepsilon }\right) ^{2}}
\leq &C_{2s+1}^{\prime }e^{-\left( \frac{2\pi \nu +\phi }{2\varepsilon }%
\right) ^{2}} \\
=&C_{2s+1}^{\prime }e^{-\left( \frac{\phi }{2\varepsilon }\right)
^{2}}e^{-\left( \frac{\pi \nu }{\varepsilon }\right) ^{2}}e^{-\frac{2\pi \nu
\phi }{4\varepsilon ^{2}}} \\
=&e^{-\left( \frac{\phi }{2\varepsilon }\right) ^{2}}C_{2s+1}^{\prime }%
\left[ e^{-\left( \frac{\pi \nu }{\varepsilon }\right) ^{2}}e^{-\frac{2\pi
\nu \phi }{4\varepsilon ^{2}}}\right] \\
\leq &e^{-\left( \frac{\phi }{2\varepsilon }\right) ^{2}}C_{2s+1}^{\prime
}e^{-\left( \frac{\pi \nu }{\varepsilon }\right) ^{2}}.
\end{align*}%
Therefore, we obtain%
\begin{align*}
V_{+} \leq &C_{2s+1}\sum_{\nu =1}^{\infty }\left\vert H_{2\left( s+1\right)
}\left( \frac{2\pi \nu +\phi }{2\varepsilon }\right) e^{-\left( \frac{2\pi
\nu +\phi }{2\varepsilon }\right) ^{2}}\right\vert  \notag \\
\leq &e^{-\left( \frac{\phi }{2\varepsilon }\right) ^{2}}C_{2s+1}\sum_{\nu
=1}^{\infty }e^{-\left( \frac{\pi \nu }{\varepsilon }\right) ^{2}}  \notag \\
\leq &e^{-\left( \frac{\phi }{2\varepsilon }\right)
^{2}}C_{2s+1}2^{2s+1}\sum_{\nu =1}^{\infty }e^{-\left( \frac{\pi \nu }{%
\varepsilon }\right) ^{2}}  \notag \\
\leq &C_{2s+1}^{\prime }e^{-\left( \frac{\phi }{2\varepsilon }\right) ^{2}}%
, 
\end{align*}%
since the series $\sum_{\nu =1}^{\infty }e^{-\left( \frac{\pi \nu }{%
\varepsilon }\right) ^{2}}$ is convergent.\\
Consider now the sum $V_{-}$. Let us define $\nu ^{\prime }=-\nu $, so that%
\begin{equation*}
V_{-}=\sum_{\nu ^{\prime }=1}^{\infty }\left\vert H_{2s+1}\left( \frac{\phi
-2\pi \nu ^{\prime }}{2\varepsilon }\right) \right\vert e^{-\left( \frac{%
\phi -2\pi \nu ^{\prime }}{2\varepsilon }\right) ^{2}}.
\end{equation*}%
Using \eqref{H_lim} yields 
\begin{align*}
\left\vert \! H_{2s+1}\left( \!\frac{\phi -2\pi \nu ^{\prime }}{2\varepsilon }%
\!\right) \!\right\vert e^{-\left( \frac{\phi -2\pi \nu ^{\prime }}{2\varepsilon 
}\right) ^{2}} \!\leq &C_{2s+1}\left\vert \frac{\phi -2\pi \nu ^{\prime }}{2\varepsilon }%
\right\vert ^{2s+1}e^{-\left( \frac{\phi -2\pi \nu ^{\prime }}{2\varepsilon }%
\right) ^{2}} \\
=&e^{-\left( \frac{\phi }{2\varepsilon }\right) ^{2}}C_{2s+1}\!\left[\!
\left\vert \frac{\phi -2\pi \nu ^{\prime }}{2\varepsilon }\right\vert
^{2s+1}\!\!e^{\frac{2\pi \nu ^{\prime }\phi }{4\varepsilon ^{2}}}e^{-\left( 
\frac{\pi \nu ^{\prime }}{\varepsilon }\right) ^{2}}\!\right] .
\end{align*}%
Since $\phi <\pi $, straightforward calculations lead to%
\begin{align}
V_{-} \leq &C_{2s+1}\sum_{\nu ^{\prime }=1}^{\infty }\left\vert
H_{2s+1}\left( \frac{\phi -2\pi \nu ^{\prime }}{2\varepsilon }\right)
\right\vert e^{-\left( \frac{\phi -2\pi \nu ^{\prime }}{2\varepsilon }%
\right) ^{2}}  \notag \\
\leq &e^{-\left( \frac{\phi }{2\varepsilon }\right) ^{2}}C_{2s+1}\sum_{\nu
^{\prime }=1}^{\infty }\left[ \left\vert \frac{\phi -2\pi \nu ^{\prime }}{%
2\varepsilon }\right\vert ^{2s+1}e^{\frac{2\pi \nu ^{\prime }\phi }{%
4\varepsilon ^{2}}}e^{-\left( \frac{\pi \nu ^{\prime }}{\varepsilon }\right)
^{2}}\right]  \notag \\
\leq &e^{-\left( \frac{\phi }{2\varepsilon }\right) ^{2}}C_{2s+1}\sum_{\nu
^{\prime }=1}^{\infty }\left[ \left\vert \frac{\pi -2\pi \nu ^{\prime }}{%
2\varepsilon }\right\vert ^{2s+1}e^{\frac{\pi ^{2}\nu ^{\prime }}{%
2\varepsilon ^{2}}}e^{-\frac{\pi ^{2}}{\varepsilon ^{2}}\left( \nu ^{\prime
}\right) ^{2}}\right]  \notag \\
\leq &e^{-\left( \frac{\phi }{2\varepsilon }\right) ^{2}}C_{2s+1}^{\prime
\prime }\sum_{\nu ^{\prime }=1}^{\infty }\left[ \left( \frac{\pi \nu
^{\prime }}{\varepsilon }\right) ^{2s+1}\exp \left( -\frac{\pi ^{2}}{%
\varepsilon ^{2}}\nu ^{\prime }\left[ \nu ^{\prime }-\frac{1}{2}\right]
\right) \right]  \notag \\
\leq &C_{2s+1}^{\prime \prime }e^{-\left( \frac{\phi }{2\varepsilon }%
\right) ^{2}}.  \label{Vmeno}
\end{align}%
Indeed, the series $\sum_{\nu ^{\prime }=1}^{\infty }\left[ \left( \frac{\pi
\nu ^{\prime }}{\varepsilon }\right) ^{2s+1}\exp \left( -\frac{\pi ^{2}}{%
\varepsilon ^{2}}\nu ^{\prime }\left[ \nu ^{\prime }-\frac{1}{2}\right]
\right) \right] $ can be proved to be convergent by means of the D'Alembert's criterion, as above.\\
\noindent Combining \eqref{Vpiu} and \eqref{Vmeno} in \eqref{Vtot}, the
term corresponding to $\nu =0$ is dominant. Hence, we have that 
\begin{align*}
V_{\varepsilon ,s}\left( \phi \right) \leq &e^{-\left( \frac{\phi }{%
2\varepsilon }\right) ^{2}}\left( \left\vert H_{2s+1}\left( \frac{\phi }{%
2\varepsilon }\right) \right\vert +C_{2s+1}^{\prime }+C_{2s+1}^{\prime
\prime }\right) \\
\leq &C_{2s+1}e^{-\left( \frac{\phi }{2\varepsilon }\right) ^{2}}\left\vert
H_{2s+1}\left( \frac{\phi }{2\varepsilon }\right) \right\vert .
\end{align*}%
Thus%
\begin{equation*}
\Lambda _{\varepsilon ;s}\left( \phi \right) \leq \frac{\widetilde{C}_{2s+1}%
}{\varepsilon ^{2}}e^{-\left( \frac{\phi }{2\varepsilon }\right)
^{2}}\left\vert H_{2s+1}\left( \frac{\phi }{2\varepsilon }\right)
\right\vert ,
\end{equation*}%
as claimed.
\end{proof}

\begin{lemma}
\label{prop:qu}For any $\vartheta \in \left[0,\pi \right]$, let $Q_{\varepsilon ,r}\left( \vartheta \right)$ and $\widetilde{Q}_{\varepsilon ,r}\left( \vartheta \right)$ be given by
\begin{align*}
Q_{\varepsilon ,r}\left( \vartheta \right) &:=\int_{\vartheta }^{\pi }\frac{%
e^{-\left( \frac{\phi }{2\varepsilon }\right) ^{2}}\left( \frac{\phi }{%
2\varepsilon }\right) ^{2r+1}}{\sqrt{\left( \phi ^{2}-\vartheta ^{2}\right) }%
}d\phi , \\
\widetilde{Q}_{\varepsilon ,r}\left( \vartheta \right) &:=\int_{0}^{%
\widetilde{\vartheta }}\frac{e^{-\left( \frac{\widetilde{\phi }}{%
2\varepsilon }\right) ^{2}}\left( \frac{\pi -\widetilde{\phi }}{2\varepsilon 
}\right) ^{2r+1}}{\sqrt{\left( \widetilde{\vartheta }^{2}-\widetilde{\phi }%
^{2}\right) }}d\widetilde{\phi }.
\end{align*}%
Then, there exist $C_{r},\widetilde{C}_{r}>0$ so that%
\begin{eqnarray*}
Q_{\varepsilon ,r}\left( \vartheta \right) &\leq &C_{r}e^{-\left( \frac{%
\vartheta }{2\varepsilon }\right) ^{2}}\left( \frac{\vartheta }{2\varepsilon 
}\right) ^{2r}, \\
\widetilde{Q}_{\varepsilon ,r}\left( \vartheta \right) &\leq &\widetilde{C}%
_{r}\left( \frac{\widetilde{\vartheta }}{2\varepsilon }\right)
^{2r}e^{-\left( \frac{\widetilde{\vartheta }}{2\varepsilon }\right) ^{2}}.
\end{eqnarray*}
\end{lemma}

\begin{proof}
First, note that 
\begin{equation*}
Q_{\varepsilon ,r}\left( \vartheta \right) =\left( \frac{\vartheta }{%
2\varepsilon }\right) ^{2r+1}\int_{\vartheta }^{\pi }\frac{e^{-\left( \frac{%
\vartheta }{2\varepsilon }\cdot \frac{\phi }{\vartheta }\right) ^{2}}\left( 
\frac{\phi }{\vartheta }\right) ^{2r+1}}{\sqrt{\left( \frac{\phi }{\vartheta 
}\right) ^{2}-1}}\frac{1}{\vartheta }d\phi \text{ .}
\end{equation*}%
Then, we use the substitution 
\begin{equation*}
t=\left( \left( \phi /\vartheta \right)
^{2}-1\right) ^{\frac{1}{2}},
\end{equation*} 
in order to obtain%
\begin{align*}
Q_{\varepsilon ,r}\left( \vartheta \right) =&e^{-\left( \frac{\vartheta }{%
2\varepsilon }\right) ^{2}}\left( \frac{\vartheta }{2\varepsilon }\right)
^{2r+1}\int_{0}^{\left( \left( \frac{\pi }{\vartheta }\right) ^{2}-1\right)
^{\frac{1}{2}}}e^{-\left( \frac{\vartheta }{2\varepsilon }\right)
^{2}t^{2}}\left( t^{2}+1\right) ^{r}dt \\
\leq &e^{-\left( \frac{\vartheta }{2\varepsilon }\right) ^{2}}\left( \frac{%
\vartheta }{2\varepsilon }\right) ^{2r+1}\int_{0}^{\infty }e^{-\left( \frac{%
\vartheta }{2\varepsilon }\right) ^{2}t^{2}}\left( t^{2}+1\right) ^{r}dt \\
=&e^{-\left( \frac{\vartheta }{2\varepsilon }\right) ^{2}}\left( \frac{%
\vartheta }{2\varepsilon }\right) ^{2r+1}\left( I_{1}+I_{2}\right) ,
\end{align*}%
where 
\begin{align*}
I_{1} &:=\int_{0}^{1}e^{-\left( \frac{\vartheta }{2\varepsilon }\right)
^{2}t^{2}}\left( t^{2}+1\right) ^{r}dt, \\
I_{2} &:=\int_{1}^{\infty }e^{-\left( \frac{\vartheta }{2\varepsilon }%
\right) ^{2}t^{2}}\left( t^{2}+1\right) ^{r}dt.
\end{align*}%
On one hand, for $t\in \left[ 0,1\right]$, we have that 
\begin{equation*}
\left( t^{2}+1\right)
^{r}\leq 2^{r}.
\end{equation*} 
On the other hand, for $t\in \left( 1,\infty \right) $, we obtain
\begin{equation*} \left(
t^{2}+1\right) ^{r}\leq \left( 2t\right) ^{2r}.
\end{equation*} 
Hence, we get 
\begin{align*}
I_{1} \leq &2^{r}\int_{0}^{1}e^{-\left( \frac{\vartheta }{2\varepsilon }%
\right) ^{2}t^{2}}dt\leq 2^{r}\int_{0}^{\infty }e^{-\left( \frac{\vartheta }{%
2\varepsilon }\right) ^{2}t^{2}}dt, \\
I_{2} \leq &\int_{1}^{\infty }e^{-\left( \frac{\vartheta }{2\varepsilon }%
\right) ^{2}t^{2}}\left( 2t\right) ^{2r}dt\leq 4^{r}\int_{0}^{\infty
}e^{-\left( \frac{\vartheta }{2\varepsilon }\right) ^{2}t^{2}}t^{2r}dt.
\end{align*}%
Straightforward calculations lead to%
\begin{align*}
I_{1}\leq &2^{r-1}\sqrt{\pi }\left( \frac{\vartheta }{2\varepsilon }\right)
^{-1},\\
I_{2}\leq &\left( \frac{\vartheta }{2\varepsilon }\right) ^{-\left(
2r+1\right) }\frac{\Gamma \left( r-\frac{1}{2}\right) }{2}.
\end{align*}%
Therefore, we obtain 
\begin{align*}
Q_{\varepsilon ,r}\left( \vartheta \right) \leq &C_{r}^{^{\prime \prime
}}e^{-\left( \frac{\vartheta }{2\varepsilon }\right) ^{2}}\left( \frac{%
\vartheta }{2\varepsilon }\right) ^{2r}\left( 1+\left( \frac{\vartheta }{%
2\varepsilon }\right) ^{-2r}\right) \\
\leq &C_{r}e^{-\left( \frac{\vartheta }{2\varepsilon }\right) ^{2}}\left( 
\frac{\vartheta }{2\varepsilon }\right) ^{2r},
\end{align*}
as claimed.\\
As far as $\widetilde{Q}_{\varepsilon ,r}\left( \vartheta \right)$ is concerned, note that the following inequality holds
\begin{align*}
\exp \left[ -\left( \frac{\pi -\widetilde{\phi }}{2\varepsilon }\right)
^{2}+\left( \frac{\widetilde{\phi }}{2\varepsilon }\right) ^{2}\right]
=&\exp \left[ -\left( \frac{\pi }{2\varepsilon }\right) ^{2}+2\frac{\pi 
\widetilde{\phi }}{2\varepsilon }\right] \\
\leq &\exp \left[ -\frac{\pi ^{2}}{2\varepsilon }\left( \frac{1}{%
2\varepsilon }-1\right) \right].
\end{align*}%
Then, let the function $\gamma \left( \cdot \right) $ on $%
\mathbb{R}$ be given by%
\begin{equation*}
\gamma \left( u\right) :=\exp \left[ -\frac{\pi ^{2}}{2u}\left( \frac{1}{2u}%
-1\right) \right].
\end{equation*}%
Note that $\gamma \left( \cdot \right)$ achieves its absolute maximum for $u=1$. Indeed, on one hand, we have that
\begin{equation*}
\gamma ^{\prime }\left( u\right) =\frac{\pi ^{2}}{2u^{2}}\left( \frac{1}{u}%
-1\right) \exp \left[ -\frac{\pi ^{2}}{2u}\left( \frac{1}{2u}-1\right) %
\right],
\end{equation*}%
so that%
\begin{equation*}
\gamma ^{\prime }\left( u\right) =0\Longleftrightarrow u=1.
\end{equation*}%
On the other hand, we have that%
\begin{equation*}
\gamma ^{\prime \prime }\left( u\right) =\frac{3\pi ^{2}}{2u^{2}}\left[ -%
\frac{1}{u^{2}}+\frac{1}{u}-\frac{1}{3}\right] \exp \left[ -\frac{\pi ^{2}}{%
2u}\left( \frac{1}{2u}-1\right) \right],
\end{equation*}%
so that%
\begin{equation*}
\gamma ^{\prime \prime }\left( 1\right) =-\frac{\pi ^{2}}{2}\exp \left( 
\frac{\pi ^{2}}{4}\right) <0.
\end{equation*}%
Finally, note that 
\begin{equation*}
\lim_{u\rightarrow \pm \infty }\gamma \left( u\right) 
=0<\exp \left( \frac{\pi
^{2}}{4}\right)=\gamma \left( 1\right) .
\end{equation*}%
Hence, we have that
\begin{equation*}
\exp \left[ -\left( \frac{\pi -\widetilde{\phi }}{2\varepsilon }\right)
^{2}+\left( \frac{\widetilde{\phi }}{2\varepsilon }\right) ^{2}\right] \leq
\exp \left[ \frac{\pi ^{2}}{4}\right]. \end{equation*}%
Furthermore, since
\begin{equation*}
\left( \frac{\pi -\widetilde{\phi }}{2\varepsilon }\right) ^{2r+1}=\left( 
\frac{\widetilde{\phi }}{2\varepsilon }\right) ^{2r+1}\left( \frac{\pi }{%
\widetilde{\phi }}-1\right) ^{2r+1}\leq \left( \frac{\widetilde{\phi }}{%
2\varepsilon }\right) ^{2r+1},
\end{equation*}%
we obtain 
\begin{equation*}
\widetilde{Q}_{\varepsilon ,r}\left( \vartheta \right) \leq \widetilde{C}%
_{r}^{\prime }\int_{0}^{\widetilde{\vartheta }}\frac{e^{-\left( \frac{%
\widetilde{\vartheta }}{2\varepsilon }\frac{\widetilde{\phi }}{\widetilde{%
\vartheta }}\right) ^{2}}\left( \frac{\widetilde{\phi }}{\widetilde{%
\vartheta }}\right) ^{2r+1}\left( \frac{\widetilde{\vartheta }}{2\varepsilon 
}\right) ^{2r+1}}{\sqrt{\left( 1-\left( \frac{\widetilde{\phi }}{\widetilde{%
\vartheta }}\right) ^{2}\right) }}\frac{1}{\widetilde{\vartheta }}d%
\widetilde{\phi }.
\end{equation*}%
Using the substitution $t=\sqrt{1-\left( \widetilde{\phi }/\widetilde{\vartheta }%
\right) ^{2}}$, we get%
\begin{align*}
\widetilde{Q}_{\varepsilon ,r}\left( \vartheta \right) \leq &\widetilde{C}%
_{r}^{\prime \prime }\left( \frac{\widetilde{\vartheta }}{2\varepsilon }%
\right) ^{2r+1}e^{-\left( \frac{\widetilde{\vartheta }}{2\varepsilon }%
\right) ^{2}}\int_{0}^{1}e^{\left( \frac{\widetilde{\vartheta }}{%
2\varepsilon }\right) ^{2}t^{2}}\left( 1-t^{2}\right) ^{r}dt \\
\leq &\widetilde{C}_{r}^{\prime \prime }\left( \frac{\widetilde{\vartheta }%
}{2\varepsilon }\right) ^{2r+1}e^{-\left( \frac{\widetilde{\vartheta }}{%
2\varepsilon }\right) ^{2}}\int_{0}^{1}e^{\left( \frac{\widetilde{\vartheta }%
}{2\varepsilon }\right) ^{2}t^{2}}dt \\
\leq &\widetilde{C}_{r}\left( \frac{\widetilde{\vartheta }}{2\varepsilon }%
\right) ^{2r}e^{-\left( \frac{\widetilde{\vartheta }}{2\varepsilon }\right)
^{2}},
\end{align*}%
as claimed.
\end{proof}

\begin{acknow*}
The author wishes to thank Federico Cacciafesta for the helpful
conversations and his enlightening suggestions and Domenico Marinucci for
his precious hints and his fundamental corrections. This paper is dedicated to Amalia Olivieri.
\end{acknow*}


\begin{thebibliography}{99}
\bibitem{abramostegun} \textbf{Abramowitz, M. and Stegun, I. (1946).} \emph{Handbook of Mathematical Functions.} Dover, New York.

\bibitem{ant} \textbf{Antoine, J.-P. and Vandergheynst, P. (2007).} Wavelets
on the Sphere and Other Conic Sections. \emph{J. Fourier Anal. Appl., 13, 4, 369--386.}

\bibitem{cammar} \textbf{Cammarota, V. and Marinucci, D. (2015).} On the Limiting
Behaviour of Needlets Polyspectra. \emph{Ann. Henri Poincar\'e Probab. Stat., 51, 3, 1159--118.}

\bibitem{da} \textbf{Dahlke, S., Steidtl, G. and Teschke, G. (2007).} Frames and
Coorbit Theory on Homogeneous Spaces with a Special Guidance on the Sphere. \emph{J. Fourier Anal. Appl., 13, 4, 387--404.}

\bibitem{dll} \textbf{Durastanti, C. and Lan, X., (2013).} High-Frequency Tail
Index Estimation by Nearly Tight Frames\textbf{, }\emph{Amer. Math. Soc. Contemp. Math., 603.}

\bibitem{tredneed} \textbf{Durastanti, C., Fantaye, Y.T., Hansen, F.K., Marinucci, D. and Pesenson, I.Z. (2014).} 
 Simple proposal for radial 3D needlets. \emph{Phys. Rev. D, 90, 103532.}
 
\bibitem{fred} \textbf{Freeden, W. and Schreiner, M. (1998).} Orthogonal and
nonorthogonal multiresolution analysis, scale discrete and exact fully
discrete wavelet transform on the sphere. \emph{Constr. Approx., 14, 4,
493--515.}

\bibitem{gmspin} \textbf{Geller, D. and Marinucci, D. (2010).} Spin Wavelets
on the Sphere. \textit{J. Fourier Anal. Appl., 16,
6, 840--884.}

\bibitem{gmmixed} \textbf{Geller, D. and Marinucci, D. (2011).} Mixed
Needlets. \emph{J. Math. Anal. Appl, 375, 2, 610--630.}

\bibitem{gm1} \textbf{Geller, D. and Mayeli, A. (2009).} Continuous Wavelets on Manifolds. \emph{Math. Z., 262, 895--927.}

\bibitem{gm2} \textbf{Geller, D. and Mayeli, A. (2009).} Nearly Tight Frames and Space-Frequency Analysis on Compact Manifolds. \emph{Math.
Z., 263, 235--264.}

\bibitem{gm3} \textbf{Geller, D. and Mayeli, A. (2009).} Besov Spaces and Frames on Compact Manifolds. \emph{Indiana Univ. Math. J., 58, 2003--2042.}

\bibitem{healpix} \textbf{Gorski, K.M., Hivon, E., Banday, A.J., Wandelt,
B.D., Hansen, F.K., Reinecke, M. and Bartelman, M. (2005).} HEALPix, a Framework
for High Resolution Discretization, and Fast Analysis of Data Distributed on
the Sphere. \emph{Astrophys. J., 622, 759--771.}

\bibitem{hol} \textbf{Holschneider, M. and Iglewska-Nowak, I. (2007).}
Poisson Wavelets on the Sphere. \emph{J. Fourier Anal. Appl., 13, 4, 405--420.}

\bibitem{knp} \textbf{Kerkyacharian, G., Nickl, R. and Picard, D. (2012).} Concentration inequalities and confidence bands for needlet density estimators on compact homogeneous manifolds. \emph{Probab. Theory Related Fields, 153, 1, 363--404.}

\bibitem{lanmar2} \textbf{Lan, X. and Marinucci, D. (2009).} On the
Dependence Structure of Wavelet Coefficients for Spherical Random Fields. 
\emph{Stochastic Process. Appl., 119, 110, 3749--3766.}

\bibitem{marpecbook} \textbf{Marinucci, D. and Peccati, G. (2011).} \emph{Random Fields on the Sphere. Representation,
Limit Theorem and Cosmological Applications.} Cambridge University Press.

\bibitem{mayeli} \textbf{Mayeli, A. (2010).} Asymptotic Uncorrelation for
Mexican Needlets. \emph{J. Math. Anal. Appl., 363, 1, 336--344.}

\bibitem{mcewen2} \textbf{McEwen, J.D., Hobson, M.P. and Lasenby A.N. (2008).} Optimal filters on the sphere. \emph{IEEE Trans. Sig. Proc., 56, 8, 3813--3823.}

\bibitem{mcewen} \textbf{McEwen, J.D., Vielva, P., Wiaux, Y., Barreiro,
R.B., Cayon, L., Hobson, M.P., Lasenby, A.N., Martinez-Gonzalez, E. and Sanz,
J. (2007).} Cosmological Applications of a Wavelet Analysis on the Sphere. 
\textit{J. Fourier Anal. Appl., 13, 495--510.}

\bibitem{mcewen3} \textbf{McEwen, J.D., Leistedt, B., B\"uttner, M., Peiris, H. V. and Wiaux, Y.(2015).} Directional spin wavelets on the sphere. \emph{IEEE Trans. Sig. Proc., submitted.}

\bibitem{mcewen4}	\textbf{McEwen, J.D., Durastanti C., and Wiaux Y. (2016).} Localisation of directional scale-discretised wavelets on the sphere. \emph{Applied Comput. Harm. Anal., in press.}

\bibitem{mcewen5} \textbf{Chan, J. Y. H., Leistedt, B., Kitching, T. D. and McEwen J. D. (2015).} Second-generation curvelets on the sphere. \emph{IEEE Trans. Sig. Proc., in press, 2015.}

\bibitem{npw1} \textbf{Narcowich, F.J., Petrushev, P. and Ward, J.D. (2006).} Localized Tight Frames on Spheres. \emph{SIAM J. Math. Anal. 38, 574--594.}

\bibitem{npw2} \textbf{Narcowich, F.J., Petrushev, P. and Ward, J.D. (2006).} Decomposition of Besov and Triebel-Lizorkin Spaces on the
Sphere. \emph{J. Funct. Anal., 238, 2, 530--564.}

\bibitem{petruxu} \textbf{Petrushev, P. and Xu,Y.(2008).} Localized Polynomial
Frames on the Ball. \emph{Constr. Approx., 27, 121--148.}

\bibitem{scodeller} \textbf{Scodeller, S., Rudjord, O. Hansen, F.K.,
Marinucci, D., Geller, D. and Mayeli, A. (2011).} Introducing Mexican
needlets for CMB analysis: Issues for practical applications and comparison
with standard needlets. \emph{Astrophys. J., 733, 2, 121.}

\bibitem{starck} \textbf{Starck, J.-L , Moudden, Y. , Abrial P. and Nguyen M. (2006).} Wavelets, Ridgelets and Curvelets on the Sphere. \emph{A\&A, 446, 1191--1204.}

\bibitem{steinweiss} \textbf{Stein, E., Weiss, G. (1971).} \emph{Introduction to Fourier Analysis on Euclidean Spaces.} Princeton University
Press.

\bibitem{szego} \textbf{Szego, G. (1975).} \emph{Orthogonal Polynomials.}
American Mathematical Society, 23.

\bibitem{wiaux} \textbf{Wiaux, Y., McEwen, J.D. and Vielva, P. (2007).}
Complex Data Processing: Fast Wavelet Analysis on the Sphere. \emph{J. Fourier Anal. Appl., 13, 4 477--494.}
\end{thebibliography}
\end{document}